\documentclass[12pt,twoside]{amsart}
\usepackage[latin1]{inputenc}
\usepackage{amsmath, amsthm, amscd, amsfonts, amssymb, graphicx}
\usepackage[bookmarksnumbered, plainpages]{hyperref}

\newtheorem{thm}{Theorem}[section]

\newtheorem{lem}[thm]{Lemma}
\newtheorem{prop}[thm]{Proposition}
\newtheorem{defn}[thm]{Definition}

\numberwithin{equation}{section}

\begin{document}

\title{{Elliptic genera and $SL(2,Z)$ modular forms for fibre bundles}}

\author{ Yong Wang\\
 }

\date{}

\thanks{{\scriptsize
\hskip -0.4 true cm \textit{2010 Mathematics Subject Classification:}
58C20; 57R20; 53C80.
\newline \textit{}}}

\maketitle

\begin{abstract}
 By the family index theory, we generalize some well-known $SL(2,Z)$ modular forms to the family case and obtain some new anomaly cancellation formulas for the determinant line bundle and index gerbes, and certain results about eta invariants. Moreover, for the higher degree case, we give some anomaly cancellation formulas of residue Chern forms.
\end{abstract}

\vskip 0.2 true cm

\section{ Introduction}
\indent In \cite{AW}, Alvarez-Gaum\'{e} and Witten discovered a formula that represents the beautiful relationship between the top
  components of the Hirzebruch $\widehat{L}$-form and $\widehat{A}$-form of a $12$-dimensional smooth Riemannian
  manifold. This formula is called the "miracle cancellation" formula for gravitational anomalies. In \cite{Li1}, Liu established higher dimensional "miraculous cancellation" formulas for $(8k+4)$-dimensional Riemannian manifolds by
  developing modular invariance properties of characteristic forms. These formulas could be used to deduce some divisibility results.
  In \cite{HLZ}, Han, Liu and Zhang showed that both of the Green-Schwarz anomaly factorization formula
for the gauge group $E_8\times E_8$ and the Horava-Witten anomaly factorization formula for the gauge
group $E_8$ could be derived through modular forms of weight $14$. This answered a question of J.
H. Schwarz. They also established generalizations of these factorization formulas and obtaind a new
Horava-Witten type factorization formula on $12$-dimensional manifolds. In \cite{HHLZ}, Han, Huang, Liu and Zhang introduced a modular form of weight $14$ over $SL(2,{\bf Z})$ and a modular form of weight $10$ over $SL(2,{\bf Z})$ and they got some interesting
anomaly cancellation formulas on $12$-dimensional manifolds. In \cite{Li2}, a modular form of weight $2k$ for a $4k$-dimensional spin manifold was introduced. In \cite{CHZ}, Chen, Han and Zhang defined an integral modular form of weight $2k$ for a $4k$-dimensional $spin^c$ manifold and an integral modular form of weight $2k$ for a $4k+2$-dimensional $spin^c$ manifold. In \cite{Wa1} and \cite{Wa2}, we obtained some new anomaly cancellation formulas for spin manifolds and spinc manifolds and complex manifolds by studying above $SL(2,{\bf Z})$ modular forms. Some divisibility results of the index of the twisted Dirac operator and spinc Dirac operators are obtained.\\
\indent On the other hand, anomaly measures the nontriviality of the determinant line bundle of a family of Dirac operators. The perturbative anomaly detects the real first Chern class of the determinant line bundle while the global anomaly detects the integral first Chern class beyond the real information. For a family of Dirac operators on an even-dimensional closed manifold, the determinant line bundle over the parametrizing space carries the Quillen metric as well as the Bismut-Freed connection compatible with the Quillen metric such that the curvature of the Bismut-Freed connection is the two-form component of the integration of the $\widehat{A}$-form of the vertical tangent bundle along the fiber in the Atiyah-Singer family index theorem \cite{BF1,BF2}. In \cite{Lo}, for a family of Dirac operators on an odd-dimensional manifold, Lott constructs an abelian gerbe-with-connection whose curvature is the three-form component of the Atiyah-Singer family index theorem. This gerbe is called the index gerbe, which is a higher analogue of the determinant line bundle. In \cite{HL}, Han and Liu by combining modoular forms and characteristic forms, they obtained anomaly cancellation formulas for the determinant line bundle and derived anomaly cancellation formulas for index gerbes and obtained certain results about eta invariants. In \cite{MP}, Mickelsson and Paycha provided local expressions for Chern-Weil forms built from superconnections associated with families of Dirac operators previously investigated by Paycha and Scott. They interpreted the even and odd higher degree Chern-Weil type forms built from superconnection as the Wodzicki residue. The purpose of this paper is to generalize some well-known $SL(2,Z)$ modular forms to the family case and obtain some new anomaly cancellation formulas for the determinant line bundle and index gerbes, and certain results about eta invariants as in \cite{HL}. Moreover, for the higher degree case, we will give some anomaly cancellation formulas of residue Chern forms.\\
\indent The structure of this paper is briefly described below: In Section 2, we generalize the results in \cite{Wa2} to the family case and obtain some new anomaly cancellation formulas for the determinant line bundle and index gerbes, and certain results about eta invariants. In Section 3, we generalize the results in \cite{Wa1} to the family case and define the family two variable elliptic genus and get some formulas for the determinant line bundle and index gerbes, and certain results about eta invariants. In Section 4, for the higher degree case, we will give some anomaly cancellation formulas of residue Chern forms by the Mickelsson and Paycha's theorem .\\
\section{Anomaly cancellation formulas and $SL(2,Z)$ modular forms for fibre bundles}
\indent Let $M$ be a closed manifold and $M\rightarrow B$ be a fibre bundle with the closed spin fibre $Z$. Let $TZ$ be the vertical bundle and $\triangle(TZ)$ be the vertical spinor bundle.
 Set
 \begin{equation}
   \Theta_1(T_{C}Z)=
   \bigotimes _{n=1}^{\infty}S_{q^n}(\widetilde{T_CZ})\otimes
\bigotimes _{m=1}^{\infty}\wedge_{q^m}(\widetilde{T_CZ})
,\end{equation}
\begin{equation}
\Theta_2(T_{C}Z)=\bigotimes _{n=1}^{\infty}S_{q^n}(\widetilde{T_CZ})\otimes
\bigotimes _{m=1}^{\infty}\wedge_{-q^{m-\frac{1}{2}}}(\widetilde{T_CZ}),
\end{equation}
\begin{equation}
\Theta_3(T_{C}Z)=\bigotimes _{n=1}^{\infty}S_{q^n}(\widetilde{T_CZ})\otimes
\bigotimes _{m=1}^{\infty}\wedge_{q^{m-\frac{1}{2}}}(\widetilde{T_CZ})
.\end{equation}
Let ${\rm dim}TZ=4k-2$ and
\begin{align}
Q(TZ,\tau)&=[\widehat{A}(TZ)[{\rm ch}(\triangle(TZ)\otimes \Theta_1(T_{C}Z))+2^{2k-1}\Theta_2(T_{C}Z)+2^{2k-1}\Theta_3(T_{C}Z)]^{(4k)}.
\end{align}
We know that the coefficient $2^{2k-1}$ is different from the coefficient $2^{2k}$ in \cite{Wa2} which makes the coefficients of our cancellation formulas are different from formulas in \cite{Wa2}.
We first recall the four Jacobi theta functions are
   defined as follows( cf. \cite{Ch}):
 \begin{equation}  \theta(v,\tau)=2q^{\frac{1}{8}}{\rm sin}(\pi
   v)\prod_{j=1}^{\infty}[(1-q^j)(1-e^{2\pi\sqrt{-1}v}q^j)(1-e^{-2\pi\sqrt{-1}v}q^j)],
   \end{equation}
\begin{equation}\theta_1(v,\tau)=2q^{\frac{1}{8}}{\rm cos}(\pi
   v)\prod_{j=1}^{\infty}[(1-q^j)(1+e^{2\pi\sqrt{-1}v}q^j)(1+e^{-2\pi\sqrt{-1}v}q^j)],\end{equation}
\begin{equation}\theta_2(v,\tau)=\prod_{j=1}^{\infty}[(1-q^j)(1-e^{2\pi\sqrt{-1}v}q^{j-\frac{1}{2}})
(1-e^{-2\pi\sqrt{-1}v}q^{j-\frac{1}{2}})],\end{equation}
\begin{equation}\theta_3(v,\tau)=\prod_{j=1}^{\infty}[(1-q^j)(1+e^{2\pi\sqrt{-1}v}q^{j-\frac{1}{2}})
(1+e^{-2\pi\sqrt{-1}v}q^{j-\frac{1}{2}})],\end{equation}
\noindent
where $q=e^{2\pi\sqrt{-1}\tau}$ with $\tau\in\textbf{H}$, the upper
half complex plane. One
has the following transformation laws of theta functions under the
actions of $S$ and $T$ (cf. \cite{Ch} ):
\begin{equation}\theta(v,\tau+1)=e^{\frac{\pi\sqrt{-1}}{4}}\theta(v,\tau),~~\theta(v,-\frac{1}{\tau})
=\frac{1}{\sqrt{-1}}\left(\frac{\tau}{\sqrt{-1}}\right)^{\frac{1}{2}}e^{\pi\sqrt{-1}\tau
v^2}\theta(\tau v,\tau);\end{equation}
\begin{equation}\theta_1(v,\tau+1)=e^{\frac{\pi\sqrt{-1}}{4}}\theta_1(v,\tau),~~\theta_1(v,-\frac{1}{\tau})
=\left(\frac{\tau}{\sqrt{-1}}\right)^{\frac{1}{2}}e^{\pi\sqrt{-1}\tau
v^2}\theta_2(\tau v,\tau);\end{equation}
\begin{equation}\theta_2(v,\tau+1)=\theta_3(v,\tau),~~\theta_2(v,-\frac{1}{\tau})
=\left(\frac{\tau}{\sqrt{-1}}\right)^{\frac{1}{2}}e^{\pi\sqrt{-1}\tau
v^2}\theta_1(\tau v,\tau);\end{equation}
\begin{equation}\theta_3(v,\tau+1)=\theta_2(v,\tau),~~\theta_3(v,-\frac{1}{\tau})
=\left(\frac{\tau}{\sqrt{-1}}\right)^{\frac{1}{2}}e^{\pi\sqrt{-1}\tau
v^2}\theta_3(\tau v,\tau),\end{equation}
 \begin{equation}\theta'(v,\tau+1)=e^{\frac{\pi\sqrt{-1}}{4}}\theta'(v,\tau),~~
 \theta'(0,-\frac{1}{\tau})=\frac{1}{\sqrt{-1}}\left(\frac{\tau}{\sqrt{-1}}\right)^{\frac{1}{2}}
\tau\theta'(0,\tau).\end{equation}
\begin{defn} A modular form over $SL(2,Z)$ is a holomorphic function $f(\tau)$ on
 $\textbf{H}$ such that
\begin{equation} f(g\tau):=f\left(\frac{a\tau+b}{c\tau+d}\right)=(c\tau+d)^lf(\tau),
 ~~\forall g=\left(\begin{array}{cc}
\ a & b  \\
 c & d
\end{array}\right)\in SL(2,Z),\end{equation}
\noindent $l$ is called the weight of $f$.
\end{defn}
\begin{thm}
$Q(TZ,\tau)$ is a modular form over $SL_2({\bf Z})$ with the weight $2k$.
\end{thm}
\begin{proof} Let $\{\pm 2\pi\sqrt{-1}x_j\}~(1\leq j \leq 2k-1)$ be the formal Chern roots for $T_CZ$, then we have
\begin{equation}Q(TZ,\tau)=\prod_{j=1}^{2k-1}\left(\frac{2x_j\theta'(0,\tau)}{\theta(x_j,\tau)}
\left(\prod_{j=1}^{2k-1}\frac{\theta_1(x_j,\tau)}{\theta_1(0,\tau)}+\prod_{j=1}^{2k-1}\frac{\theta_2(x_j,\tau)}
{\theta_2(0,\tau)}+\prod_{j=1}^{2k-1}\frac{\theta_3(x_j,\tau)}{\theta_3(0,\tau)}
\right)\right).\end{equation}
By (2.9)-(2.13), we have $Q(TZ,\tau+1)=Q(TZ,\tau)$ and $Q(TZ,-\frac{1}{\tau})=\tau^{2k}Q(TZ,\tau)$, so
$Q(TZ,\tau)$ is a modular form over $SL_2({\bf Z})$ with the weight $2k$.
\end{proof}
We expand $Q(TZ,\tau)$ as follows:
\begin{align}
&Q(TZ,\tau)=\left[\widehat{A}(TZ){\rm ch}(\triangle(TZ))+2^{2k}\widehat{A}(TZ)\right]^{(4k)}\\\notag
&+\left[2\widehat{A}(TZ){\rm ch}(\triangle(TZ)){\rm ch}(\widetilde{T_CZ})+2^{2k}\widehat{A}(TZ){\rm ch} (\widetilde{T_CZ}+\wedge^2\widetilde{T_CZ})\right]^{(4k)}q\\\notag
&+\left[\widehat{A}(TZ){\rm ch}(\triangle(TZ)){\rm ch}(A_0)
+2^{2k}\widehat{A}(TZ){\rm ch}(A_1)\right]^{(4k)}q^2+\cdots,
\end{align}
where
\begin{align}
&A_0=2\widetilde{T_CZ}+\wedge^2\widetilde{T_CZ}+\widetilde{T_CZ}\otimes \widetilde{T_CZ}+S^2\widetilde{T_CZ},\\\notag
&A_1=\wedge^4\widetilde{T_CZ}+\wedge^2\widetilde{T_CZ}\otimes \widetilde{T_CZ}
+\widetilde{T_CZ}\otimes \widetilde{T_CZ}+S^2\widetilde{T_CZ}+
\widetilde{T_CZ}.
\end{align}
Let $\mathcal{L}_{D^Z\otimes V}=\det({\rm Ker}(D^Z\otimes V)_+)^*\otimes\det({\rm Ker}(D^Z\otimes V)_-)$ be the determinant line bundle of the family operator $D^Z\otimes V$ over $Z$ (for associated definitions, see \cite{BF1},\cite{BF2}). The determinant line bundle carries the Quillen metric $g^{\mathcal{L}_{D^Z\otimes V}}$ as well as the Bismut-Freed connection $\nabla^{\mathcal{L}_{D^Z\otimes V}}$ compatible to $g^{\mathcal{L}_{D^Z\otimes V}}$, the curvature $R^{\mathcal{L}_{D^Z\otimes V}}$ of which
is equal to the two-form component of the integration of $\widehat{A}(TZ,\nabla^{TZ}){\rm ch}(V,\nabla^V)$ along the fibre in the Atiyah-Singer family index theorem \cite{BF1,BF2}. $\sqrt{-1}/(2\pi)R^{\mathcal{L}_{D^Z\otimes V}}$ is a representative of the local anomaly.
\begin{thm}
{\rm(Bismut and Freed \cite{BF2})}
  \begin{equation}
    \mathbf{R}^{\mathcal{L}_{D^Z\otimes V}}=2\pi\sqrt{-1}\left\{\int_Z\widehat{A}(TZ,\nabla^{Z}){\rm ch}(V,\nabla^V)\right\}^{(2)}.
  \end{equation}
\end{thm}
The Bismut-Freed theorem is correct for the spinc case.

\begin{thm}
When ${\rm dim}TZ=6$, we have
\begin{align}
&\mathbf{R}^{\mathcal{L}_{D^Z\otimes \triangle(TZ)\otimes\widetilde{T_CZ}}}
+8\mathbf{R}^{\mathcal{L}_{D^Z\otimes (\widetilde{T_CZ}+\wedge^2\widetilde{T_CZ}) }}
=120(\mathbf{R}^{\mathcal{L}_{D^Z\otimes \triangle(TZ)}}+16\mathbf{R}^{\mathcal{L}_{D^Z}}),
\end{align}
\begin{align}
&\mathbf{R}^{\mathcal{L}_{D^Z\otimes \triangle(TZ)\otimes A_0}}
+16\mathbf{R}^{\mathcal{L}_{D^Z\otimes A_1}}
=2160(\mathbf{R}^{\mathcal{L}_{D^Z\otimes \triangle(TZ)}}+16\mathbf{R}^{\mathcal{L}_{D^Z}}).
\end{align}
\end{thm}

{\bf The proof of Theorem 2.4:}
It is well known that modular forms over $SL_2({\bf Z})$ can be expressed as polynomials of the Einsentein series $E_4(\tau)$ and $E_6(\tau)$,
where
 \begin{equation}
E_4(\tau)=1+240q+2160q^2+6720q^3+\cdots,
\end{equation}
\begin{equation}
E_6(\tau)=1-504q-16632q^2-122976q^3+\cdots.
\end{equation}
Their weights are $4$ and $6$ respectively. When $k=2$, then $Q(TZ,\tau)$ is a modular form over $SL_2({\bf Z})$ with the weight $4$. $Q(TZ,\tau)$ must be a multiple of
\begin{equation}
Q(TZ,\tau)=\lambda E_4(\tau).
\end{equation}
By Theorem 2.2 and 2.3 and (2.16-2.17), we have Theorem 2.4.
Similarly to Theorem 2.4, by the method in \cite{Wa2}, we have

\begin{thm}
When ${\rm dim}TZ=10$, we have
\begin{align}
&\mathbf{R}^{\mathcal{L}_{D^Z\otimes \triangle(TZ)\otimes\widetilde{T_CZ}}}
+32\mathbf{R}^{\mathcal{L}_{D^Z\otimes (\widetilde{T_CZ}+\wedge^2\widetilde{T_CZ}) }}
=-252(\mathbf{R}^{\mathcal{L}_{D^Z\otimes \triangle(TZ)}}+64\mathbf{R}^{\mathcal{L}_{D^Z}}),
\end{align}
\begin{align}
&\mathbf{R}^{\mathcal{L}_{D^Z\otimes \triangle(TZ)\otimes A_0}}
+64\mathbf{R}^{\mathcal{L}_{D^Z\otimes A_1}}
=-16632(\mathbf{R}^{\mathcal{L}_{D^Z\otimes \triangle(TZ)}}+64\mathbf{R}^{\mathcal{L}_{D^Z}}).
\end{align}
\end{thm}

\begin{thm}
When ${\rm dim}TZ=14$, we have
\begin{align}
&\mathbf{R}^{\mathcal{L}_{D^Z\otimes \triangle(TZ)\otimes\widetilde{T_CZ}}}
+128\mathbf{R}^{\mathcal{L}_{D^Z\otimes (\widetilde{T_CZ}+\wedge^2\widetilde{T_CZ}) }}
=240(\mathbf{R}^{\mathcal{L}_{D^Z\otimes \triangle(TZ)}}+256\mathbf{R}^{\mathcal{L}_{D^Z}}),
\end{align}
\begin{align}
&\mathbf{R}^{\mathcal{L}_{D^Z\otimes \triangle(TZ)\otimes A_0}}
+256\mathbf{R}^{\mathcal{L}_{D^Z\otimes A_1}}
=61920(\mathbf{R}^{\mathcal{L}_{D^Z\otimes \triangle(TZ)}}+256\mathbf{R}^{\mathcal{L}_{D^Z}}).
\end{align}
\end{thm}

\begin{thm}
When ${\rm dim}TZ=18$, we have
\begin{align}
&\mathbf{R}^{\mathcal{L}_{D^Z\otimes \triangle(TZ)\otimes\widetilde{T_CZ}}}
+512\mathbf{R}^{\mathcal{L}_{D^Z\otimes (\widetilde{T_CZ}+\wedge^2\widetilde{T_CZ}) }}
=-132(\mathbf{R}^{\mathcal{L}_{D^Z\otimes \triangle(TZ)}}+1024\mathbf{R}^{\mathcal{L}_{D^Z}}),
\end{align}
\begin{align}
&\mathbf{R}^{\mathcal{L}_{D^Z\otimes \triangle(TZ)\otimes A_0}}
+1024\mathbf{R}^{\mathcal{L}_{D^Z\otimes A_1}}
=-117288(\mathbf{R}^{\mathcal{L}_{D^Z\otimes \triangle(TZ)}}+1024\mathbf{R}^{\mathcal{L}_{D^Z}}).
\end{align}
\end{thm}

Let ${\rm dim}TZ=4k-3$ and
\begin{align}
\widetilde{Q}(TZ,\tau)&=[\widehat{A}(TZ)[{\rm ch}(\triangle(TZ)\otimes \Theta_1(T_{C}Z))+2^{2k-2}\Theta_2(T_{C}Z)+2^{2k-2}\Theta_3(T_{C}Z)]^{(4k)}.
\end{align}
\begin{thm}
$\widetilde{Q}(TZ,\tau)$ is a modular form over $SL_2({\bf Z})$ with the weight $2k$.
\end{thm}

We expand $\widetilde{Q}(TZ,\tau)$ as follows:
\begin{align}
&\widetilde{Q}(TZ,\tau)=\left[\widehat{A}(TZ){\rm ch}(\triangle(TZ))+2^{2k-1}\widehat{A}(TZ)\right]^{(4k)}\\\notag
&+\left[2\widehat{A}(TZ){\rm ch}(\triangle(TZ)){\rm ch}(\widetilde{T_CZ})+2^{2k-1}\widehat{A}(TZ){\rm ch} (\widetilde{T_CZ}+\wedge^2\widetilde{T_CZ})\right]^{(4k)}q\\\notag
&+\left[\widehat{A}(TZ){\rm ch}(\triangle(TZ)){\rm ch}(A_0)
+2^{2k-1}\widehat{A}(TZ){\rm ch}(A_1)\right]^{(4k)}q^2+\cdots.
\end{align}

Below we consider anomaly cancellation formulas of index gerbes.

Now we still assume that $TZ$ is oriented but the dimension of the fibre is $4k-3$, i.e. we
consider odd-dimensional fibres. Let $TZ$ is spin. One can still define the family of Dirac operator
$D^Z\otimes V$ similar as the even-dimensional fibre case. The only difference is that now the spinor bundle $\triangle(TZ)$ associated to $TZ$ is not $\mathbb{Z}_2$-graded.  Let $H^{\infty}$ be the set of $C^{\infty}$ sections of $\triangle(TZ)\otimes V$ over $M$. $H^{\infty}$ is viewed as the set of $C^{\infty}$ sections over $B$ of infinite-dimensional bundles which are still denoted by $H^{\infty}$. For $b\in B$,~$H^{\infty}_b$ is the set of $C^{\infty}$ sections over $Z_b$ of $\triangle(TZ)\otimes V$. For each $b\in B$,
$$(D^Z\otimes V)_b\in{\rm End}(H^{\infty}_b)$$
is the twisted Dirac operator on the fibre $Z_b$.

As a higher analogue of the determinant line bundle, Lott \cite{Lo} constructs the index gerbe $\mathcal{G}_{D^Z\otimes V}$ with connection on $B$ for the family twisted odd Dirac operator $D^Z\otimes V$, the curvature $R^{\mathcal{G}_{D^Z\otimes V}}$ (a closed 3-form on $B$) of which is equal to the three-form component of the integration of $\widehat{A}(TZ,\nabla^Z){\rm ch}(V,\nabla^V)$ along the fibre in the Atiyah-Singer families index theorem.
\begin{thm}
{\rm(Lott \cite{Lo})}
 \begin{equation}
    \mathbf{R}^{\mathcal{G}_{D^Z\otimes V}}=\left\{\int_Z\widehat{A}(TZ,\nabla^{Z}){\rm ch}(V,\nabla^V)\right\}^{(3)}.
  \end{equation}
\end{thm}

Similarly to Theorem 2.4-2.7, We have the following anomaly cancellation formulas for index gerbes.
\begin{thm}
When ${\rm dim}TZ=5$, we have
\begin{align}
&\mathbf{R}^{\mathcal{G}_{D^Z\otimes \triangle(TZ)\otimes\widetilde{T_CZ}}}
+4\mathbf{R}^{\mathcal{L}_{D^Z\otimes (\widetilde{T_CZ}+\wedge^2\widetilde{T_CZ}) }}
=120(\mathbf{R}^{\mathcal{G}_{D^Z\otimes \triangle(TZ)}}+8\mathbf{G}^{\mathcal{L}_{D^Z}}),
\end{align}
\begin{align}
&\mathbf{R}^{\mathcal{G}_{D^Z\otimes \triangle(TZ)\otimes A_0}}
+8\mathbf{R}^{\mathcal{L}_{D^Z\otimes A_1}}
=2160(\mathbf{R}^{\mathcal{G}_{D^Z\otimes \triangle(TZ)}}+8\mathbf{R}^{\mathcal{G}_{D^Z}}).
\end{align}
\end{thm}

\begin{thm}
When ${\rm dim}TZ=9$, we have
\begin{align}
&\mathbf{R}^{\mathcal{G}_{D^Z\otimes \triangle(TZ)\otimes\widetilde{T_CZ}}}
+16\mathbf{R}^{\mathcal{L}_{D^Z\otimes (\widetilde{T_CZ}+\wedge^2\widetilde{T_CZ}) }}
=-252(\mathbf{R}^{\mathcal{G}_{D^Z\otimes \triangle(TZ)}}+32\mathbf{G}^{\mathcal{L}_{D^Z}}),
\end{align}
\begin{align}
&\mathbf{R}^{\mathcal{G}_{D^Z\otimes \triangle(TZ)\otimes A_0}}
+32\mathbf{R}^{\mathcal{L}_{D^Z\otimes A_1}}
=-16632(\mathbf{R}^{\mathcal{G}_{D^Z\otimes \triangle(TZ)}}+32\mathbf{R}^{\mathcal{G}_{D^Z}}).
\end{align}
\end{thm}

\begin{thm}
When ${\rm dim}TZ=13$, we have
\begin{align}
&\mathbf{R}^{\mathcal{G}_{D^Z\otimes \triangle(TZ)\otimes\widetilde{T_CZ}}}
+64\mathbf{R}^{\mathcal{L}_{D^Z\otimes (\widetilde{T_CZ}+\wedge^2\widetilde{T_CZ}) }}
=240(\mathbf{R}^{\mathcal{G}_{D^Z\otimes \triangle(TZ)}}+128\mathbf{G}^{\mathcal{L}_{D^Z}}),
\end{align}
\begin{align}
&\mathbf{R}^{\mathcal{G}_{D^Z\otimes \triangle(TZ)\otimes A_0}}
+128\mathbf{R}^{\mathcal{L}_{D^Z\otimes A_1}}
=61920(\mathbf{R}^{\mathcal{G}_{D^Z\otimes \triangle(TZ)}}+128\mathbf{R}^{\mathcal{G}_{D^Z}}).
\end{align}
\end{thm}

 \begin{thm}
When ${\rm dim}TZ=17$, we have
\begin{align}
&\mathbf{R}^{\mathcal{G}_{D^Z\otimes \triangle(TZ)\otimes\widetilde{T_CZ}}}
+256\mathbf{R}^{\mathcal{L}_{D^Z\otimes (\widetilde{T_CZ}+\wedge^2\widetilde{T_CZ}) }}
=-132(\mathbf{R}^{\mathcal{G}_{D^Z\otimes \triangle(TZ)}}+512\mathbf{G}^{\mathcal{L}_{D^Z}}),
\end{align}
\begin{align}
&\mathbf{R}^{\mathcal{G}_{D^Z\otimes \triangle(TZ)\otimes A_0}}
+512\mathbf{R}^{\mathcal{L}_{D^Z\otimes A_1}}
=-117288(\mathbf{R}^{\mathcal{G}_{D^Z\otimes \triangle(TZ)}}+512\mathbf{R}^{\mathcal{G}_{D^Z}}).
\end{align}
\end{thm}

Let ${\rm dim}TZ=4k-1$ and define
$\widehat{Q}(TZ,\tau)$ which has same expression and expansion as
${Q}(TZ,\tau)$. $\widehat{Q}(TZ,\tau)$ is a modular form over $SL_2({\bf Z})$ with the weight $2k$.
Below we consider $\eta$-invariants.

For $b\in B$, let $\eta_b(D^Z\otimes V)(s)$ be the eta function associated with $(D^Z\otimes V)_b$. Define as Atiyah, Patodi and Singer in \cite{APS}
$$\bar\eta_b(D^Z\otimes V)(s)=\frac{\eta_b(D^Z\otimes V)(s)+\ker(D^Z\otimes V)_b}{2}.$$
Denote $\bar\eta_b(D^Z\otimes V)(0)$ (a function on $B$) by $\bar\eta(D^Z\otimes V)$.

We still adopt the setting of family twisted Dirac operators on a family of odd-dimensional manifolds in index gerbes. For odd-dimensional fibres, we have the following Bismut-Freed theorem for the reduced $\eta$-invariants.
\begin{thm}
{\rm(Bismut and Freed \cite{BF2})}
 \begin{equation}
    d\{\bar\eta(D^Z\otimes V)\}=\left\{\int_Z\widehat{A}(TZ,\nabla^{Z}){\rm ch}(V,\nabla^V)\right\}^{(1)}.
  \end{equation}
\end{thm}

We have the following theorems on the reduced $\eta$-invariants.

\begin{thm}
When ${\rm dim}TZ=7$, we have
\begin{align}
&\bar\eta(D^Z\otimes \triangle(TZ)\otimes\widetilde{T_CZ})
+8\bar\eta(D^Z\otimes (\widetilde{T_CZ}+\wedge^2\widetilde{T_CZ}))\\\notag
&=120(\bar\eta(D^Z\otimes \triangle(TZ))+16\bar\eta(D^Z))+c_1,
\end{align}
\begin{align}
&\bar\eta(D^Z\otimes \triangle(TZ)\otimes A_0)
+16\bar\eta(D^Z\otimes A_1)
=2160(\bar\eta(D^Z\otimes \triangle(TZ))+16\bar\eta(D^Z))+c_2.
\end{align}
\end{thm}

\begin{thm}
When ${\rm dim}TZ=11$, we have
\begin{align}
&\bar\eta(D^Z\otimes \triangle(TZ)\otimes\widetilde{T_CZ})
+32\bar\eta(D^Z\otimes (\widetilde{T_CZ}+\wedge^2\widetilde{T_CZ}))\\\notag
&=-252(\bar\eta(D^Z\otimes \triangle(TZ))+64\bar\eta(D^Z))+c_3,
\end{align}
\begin{align}
&\bar\eta(D^Z\otimes \triangle(TZ)\otimes A_0)
+64\bar\eta(D^Z\otimes A_1)
=-16632(\bar\eta(D^Z\otimes \triangle(TZ))+64\bar\eta(D^Z))+c_4.
\end{align}
\end{thm}

\begin{thm}
When ${\rm dim}TZ=15$, we have
\begin{align}
&\bar\eta(D^Z\otimes \triangle(TZ)\otimes\widetilde{T_CZ})
+128\bar\eta(D^Z\otimes (\widetilde{T_CZ}+\wedge^2\widetilde{T_CZ}))\\\notag
&=240(\bar\eta(D^Z\otimes \triangle(TZ))+256\bar\eta(D^Z))+c_5,
\end{align}
\begin{align}
&\bar\eta(D^Z\otimes \triangle(TZ)\otimes A_0)
+256\bar\eta(D^Z\otimes A_1)
=61920(\bar\eta(D^Z\otimes \triangle(TZ))+256\bar\eta(D^Z))+c_6.
\end{align}
\end{thm}

\begin{thm}
When ${\rm dim}TZ=19$, we have
\begin{align}
&\bar\eta(D^Z\otimes \triangle(TZ)\otimes\widetilde{T_CZ})
+512\bar\eta(D^Z\otimes (\widetilde{T_CZ}+\wedge^2\widetilde{T_CZ}))\\\notag
&=-132(\bar\eta(D^Z\otimes \triangle(TZ))+1024\bar\eta(D^Z))+c_7,
\end{align}
\begin{align}
&\bar\eta(D^Z\otimes \triangle(TZ)\otimes A_0)
+1024\bar\eta(D^Z\otimes A_1)
=-117288(\bar\eta(D^Z\otimes \triangle(TZ))+1024\bar\eta(D^Z))+c_{8},
\end{align}
where $c_j,~1\leq j\leq 8$ is a constant.
\end{thm}

Let $TZ$ be ${\rm spin^{c}}$ and $L$ be the complex line bundle associated to the given ${\rm spin^{c}}$ structure on $TZ$. Denote by $c=c_1(L)$ the first Chern class of $L.$ Also, we use $L_{\bf{R}}$ for the notation of $L,$ when it is viewed as an oriented real plane bundle.
Let $\Theta(T_{\mathbf{C}}Z,L_{\bf{R}}\otimes\bf{C})$ be the virtual complex vector bundle over $M$ defined by
\begin{equation}
    \begin{split}
        \Theta(T_{\mathbf{C}}M,L_{\bf{R}}\otimes\mathbf{C})=&\bigotimes _{n=1}^{\infty}S_{q^n}(\widetilde{T_{\mathbf{C}}M})\otimes
\bigotimes _{m=1}^{\infty}\wedge_{q^m}(\widetilde{L_{\bf{R}}\otimes\mathbf{C}})\\
&\otimes
\bigotimes _{r=1}^{\infty}\wedge_{-q^{r-\frac{1}{2}}}(\widetilde{L_{\bf{R}}\otimes\mathbf{C}})\otimes
\bigotimes _{s=1}^{\infty}\wedge_{q^{s-\frac{1}{2}}}(\widetilde{L_{\bf{R}}\otimes\mathbf{C}}),\nonumber
    \end{split}
\end{equation}

\begin{equation}
 \Theta^*(T_{\mathbf{C}}M,L_{\bf{R}}\otimes\mathbf{C})=\bigotimes _{n=1}^{\infty}S_{q^n}(\widetilde{T_{\mathbf{C}}M})\otimes
\bigotimes _{m=1}^{\infty}\wedge_{-q^m}(\widetilde{L_{\bf{R}}\otimes\mathbf{C}}).\\\notag
\end{equation}
For ${\rm dim}TZ=4k-3,~4k-2,~4k-1$ three cases, we always let
\begin{equation}
Q(TZ,L,\tau)
=\{\widehat{A}(TZ,\nabla^{TZ}){\rm exp}(\frac{c}{2}){\rm ch}(\Theta(T_{\mathbf{C}}Z,L_{\bf{R}}\otimes\mathbf{C}))\}^{(4k)},
\end{equation}
\begin{equation}
Q^*(TZ,L,\tau)
=\{\widehat{A}(TZ,\nabla^{TZ}){\rm exp}(\frac{c}{2}){\rm ch}(\Theta^*(T_{\mathbf{C}}Z,L_{\bf{R}}\otimes\mathbf{C}))\}^{(4k+2)}.
\end{equation}
\begin{prop}
When $p_1(TZ)=3p_1(L_R)$,then $Q(TZ,L,\tau)$ is a modular form over $SL_2({\bf Z})$ with the weight $2k$. When $p_1(TZ)=p_1(L_R)$,then $Q^*(TZ,L,\tau)$ is a modular form over $SL_2({\bf Z})$ with the weight $2k$.
\end{prop}
We have
\begin{align}
&Q(TZ,L,\tau)=
\left\{\widehat{A}(TZ){\rm exp}(\frac{c}{2}))\right\}^{(4k)}
+q\left\{\widehat{A}(TZ){\rm exp}(\frac{c}{2})){\rm ch}(B_1)\right\}^{(4k)}\\\notag
&+q^2\left\{\widehat{A}(TZ){\rm exp}(\frac{c}{2})){\rm ch}(B_2)\right\}^{(4k)}+O(q^3),
\end{align}
where
\begin{align}
&B_1=\widetilde{T_CZ}+2\wedge^2\widetilde{L_{R,C}}-\widetilde{L_{R,C}}\otimes \widetilde{L_{R,C}}
+\widetilde{L_{R,C}},\\\notag
&B_2=
S^2\widetilde{T_CZ}+\widetilde{T_CZ}+
(2\wedge^2\widetilde{L_{R,C}}-\widetilde{L_{R,C}}\otimes \widetilde{L_{R,C}}
+\widetilde{L_{R,C}})\otimes\widetilde{T_CZ}
+\wedge^2\widetilde{L_{R,C}}\otimes\wedge^2\widetilde{L_{R,C}}\\\notag
&+2\wedge^4\widetilde{L_{R,C}}-2\widetilde{L_{R,C}}\otimes \wedge^3\widetilde{L_{R,C}}+
2\widetilde{L_{R,C}}\otimes \wedge^2\widetilde{L_{R,C}}-\widetilde{L_{R,C}}\otimes \widetilde{L_{R,C}}\otimes \widetilde{L_{R,C}}
+\widetilde{L_{R,C}}+ \wedge^2\widetilde{L_{R,C}}.\\\notag
\end{align}
\begin{align}
&Q^*(TZ,L,\tau)=
\left\{\widehat{A}(TZ){\rm exp}(\frac{c}{2}))\right\}^{(4k+2)}
+q\left\{\widehat{A}(TZ){\rm exp}(\frac{c}{2})){\rm ch}(B_3)\right\}^{(4k+2)}\\\notag
&+q^2\left\{\widehat{A}(TZ){\rm exp}(\frac{c}{2})){\rm ch}(B_4)\right\}^{(4k+2)}+O(q^3),
\end{align}
where
\begin{align}
&B_3=\widetilde{T_CZ}-\widetilde{L_{R,C}},\\\notag
&B_4=\wedge^2\widetilde{L_{R,C}}-\widetilde{L_{R,C}}-\widetilde{L_{R,C}}\otimes\widetilde{T_CZ}+S^2\widetilde{T_CZ}+\widetilde{T_CZ}
.\\\notag
\end{align}
Denote the spinc Dirac operator by $D_c$. Similarly to previous theorems, we have

\begin{thm}
When ${\rm dim}TZ=6$ and $p_1(TZ)=3p_1(L_R)$, we have
\begin{align}
&\mathbf{R}^{\mathcal{L}_{D_c^Z\otimes B_1}}=240\mathbf{R}^{\mathcal{L}_{D^Z_c}},\\\notag
&\mathbf{R}^{\mathcal{L}_{D_c^Z\otimes B_2}}=2160\mathbf{R}^{\mathcal{L}_{D^Z_c}}.
\end{align}
When ${\rm dim}TZ=6$ and $p_1(TZ)=p_1(L_R)$, we have
\begin{align}
&\mathbf{R}^{\mathcal{L}_{D_c^Z\otimes B_3}}=240\mathbf{R}^{\mathcal{L}_{D^Z_c}},\\\notag
&\mathbf{R}^{\mathcal{L}_{D_c^Z\otimes B_4}}=2160\mathbf{R}^{\mathcal{L}_{D^Z_c}}.
\end{align}
\end{thm}
\begin{thm}
When ${\rm dim}TZ=5$ and $p_1(TZ)=3p_1(L_R)$, we have
\begin{align}
&\mathbf{R}^{\mathcal{G}_{D_c^Z\otimes B_1}}=240\mathbf{R}^{\mathcal{G}_{D^Z_c}},\\\notag
&\mathbf{R}^{\mathcal{G}_{D_c^Z\otimes B_2}}=2160\mathbf{R}^{\mathcal{G}_{D^Z_c}}.
\end{align}
When ${\rm dim}TZ=5$ and $p_1(TZ)=p_1(L_R)$, we have
\begin{align}
&\mathbf{R}^{\mathcal{G}_{D_c^Z\otimes B_3}}=240\mathbf{R}^{\mathcal{G}_{D^Z_c}},\\\notag
&\mathbf{R}^{\mathcal{G}_{D_c^Z\otimes B_4}}=2160\mathbf{R}^{\mathcal{G}_{D^Z_c}}.
\end{align}
\end{thm}

\begin{thm}
When ${\rm dim}TZ=7$ and $p_1(TZ)=3p_1(L_R)$, we have
\begin{align}
&\overline{\eta}(D_c^Z\otimes B_1)=240\overline{\eta}(D^Z_c)+c_1,\\\notag
&\overline{\eta}(D_c^Z\otimes B_2)=2160\overline{\eta}(D^Z_c)+c_2.
\end{align}
When ${\rm dim}TZ=7$ and $p_1(TZ)=p_1(L_R)$, we have
\begin{align}
&\overline{\eta}(D_c^Z\otimes B_3)=240\overline{\eta}(D^Z_c)+c_3,\\\notag
&\overline{\eta}(D_c^Z\otimes B_4)=2160\overline{\eta}(D^Z_c)+c_4.
\end{align}
\end{thm}

\begin{thm}
When ${\rm dim}TZ=10$ and $p_1(TZ)=3p_1(L_R)$, we have
\begin{align}
&\mathbf{R}^{\mathcal{L}_{D_c^Z\otimes B_1}}=-504\mathbf{R}^{\mathcal{L}_{D^Z_c}},\\\notag
&\mathbf{R}^{\mathcal{L}_{D_c^Z\otimes B_2}}=-16632\mathbf{R}^{\mathcal{L}_{D^Z_c}}.
\end{align}
When ${\rm dim}TZ=10$ and $p_1(TZ)=p_1(L_R)$, we have
\begin{align}
&\mathbf{R}^{\mathcal{L}_{D_c^Z\otimes B_3}}=-504\mathbf{R}^{\mathcal{L}_{D^Z_c}},\\\notag
&\mathbf{R}^{\mathcal{L}_{D_c^Z\otimes B_4}}=-16632\mathbf{R}^{\mathcal{L}_{D^Z_c}}.
\end{align}
\end{thm}
\begin{thm}
When ${\rm dim}TZ=9$ and $p_1(TZ)=3p_1(L_R)$, we have
\begin{align}
&\mathbf{R}^{\mathcal{G}_{D_c^Z\otimes B_1}}=-504\mathbf{R}^{\mathcal{G}_{D^Z_c}},\\\notag
&\mathbf{R}^{\mathcal{G}_{D_c^Z\otimes B_2}}=-16632\mathbf{R}^{\mathcal{G}_{D^Z_c}}.
\end{align}
When ${\rm dim}TZ=9$ and $p_1(TZ)=p_1(L_R)$, we have
\begin{align}
&\mathbf{R}^{\mathcal{G}_{D_c^Z\otimes B_3}}=-504\mathbf{R}^{\mathcal{G}_{D^Z_c}},\\\notag
&\mathbf{R}^{\mathcal{G}_{D_c^Z\otimes B_4}}=-16632\mathbf{R}^{\mathcal{G}_{D^Z_c}}.
\end{align}
\end{thm}

\begin{thm}
When ${\rm dim}TZ=11$ and $p_1(TZ)=3p_1(L_R)$, we have
\begin{align}
&\overline{\eta}(D_c^Z\otimes B_1)=-504\overline{\eta}(D^Z_c)+c_5,\\\notag
&\overline{\eta}(D_c^Z\otimes B_2)=-16632\overline{\eta}(D^Z_c)+c_6.
\end{align}
When ${\rm dim}TZ=11$ and $p_1(TZ)=p_1(L_R)$, we have
\begin{align}
&\overline{\eta}(D_c^Z\otimes B_3)=-504\overline{\eta}(D^Z_c)+c_7,\\\notag
&\overline{\eta}(D_c^Z\otimes B_4)=-16632\overline{\eta}(D^Z_c)+c_8.
\end{align}
\end{thm}

\begin{thm}
When ${\rm dim}TZ=14$ and $p_1(TZ)=3p_1(L_R)$, we have
\begin{align}
&\mathbf{R}^{\mathcal{L}_{D_c^Z\otimes B_1}}=480\mathbf{R}^{\mathcal{L}_{D^Z_c}},\\\notag
&\mathbf{R}^{\mathcal{L}_{D_c^Z\otimes B_2}}=61920\mathbf{R}^{\mathcal{L}_{D^Z_c}}.
\end{align}
When ${\rm dim}TZ=14$ and $p_1(TZ)=p_1(L_R)$, we have
\begin{align}
&\mathbf{R}^{\mathcal{L}_{D_c^Z\otimes B_3}}=480\mathbf{R}^{\mathcal{L}_{D^Z_c}},\\\notag
&\mathbf{R}^{\mathcal{L}_{D_c^Z\otimes B_4}}=61920\mathbf{R}^{\mathcal{L}_{D^Z_c}}.
\end{align}
\end{thm}
\begin{thm}
When ${\rm dim}TZ=13$ and $p_1(TZ)=3p_1(L_R)$, we have
\begin{align}
&\mathbf{R}^{\mathcal{G}_{D_c^Z\otimes B_1}}=480\mathbf{R}^{\mathcal{G}_{D^Z_c}},\\\notag
&\mathbf{R}^{\mathcal{G}_{D_c^Z\otimes B_2}}=61920\mathbf{R}^{\mathcal{G}_{D^Z_c}}.
\end{align}
When ${\rm dim}TZ=13$ and $p_1(TZ)=p_1(L_R)$, we have
\begin{align}
&\mathbf{R}^{\mathcal{G}_{D_c^Z\otimes B_3}}=480\mathbf{R}^{\mathcal{G}_{D^Z_c}},\\\notag
&\mathbf{R}^{\mathcal{G}_{D_c^Z\otimes B_4}}=61920\mathbf{R}^{\mathcal{G}_{D^Z_c}}.
\end{align}
\end{thm}

\begin{thm}
When ${\rm dim}TZ=15$ and $p_1(TZ)=3p_1(L_R)$, we have
\begin{align}
&\overline{\eta}(D_c^Z\otimes B_1)=480\overline{\eta}(D^Z_c)+c_9,\\\notag
&\overline{\eta}(D_c^Z\otimes B_2)=61920\overline{\eta}(D^Z_c)+c_{10}.
\end{align}
When ${\rm dim}TZ=15$ and $p_1(TZ)=p_1(L_R)$, we have
\begin{align}
&\overline{\eta}(D_c^Z\otimes B_3)=480\overline{\eta}(D^Z_c)+c_{11},\\\notag
&\overline{\eta}(D_c^Z\otimes B_4)=61920\overline{\eta}(D^Z_c)+c_{12}.
\end{align}
\end{thm}

\begin{thm}
When ${\rm dim}TZ=18$ and $p_1(TZ)=3p_1(L_R)$, we have
\begin{align}
&\mathbf{R}^{\mathcal{L}_{D_c^Z\otimes B_1}}=-264\mathbf{R}^{\mathcal{L}_{D^Z_c}},\\\notag
&\mathbf{R}^{\mathcal{L}_{D_c^Z\otimes B_2}}=-117288\mathbf{R}^{\mathcal{L}_{D^Z_c}}.
\end{align}
When ${\rm dim}TZ=18$ and $p_1(TZ)=p_1(L_R)$, we have
\begin{align}
&\mathbf{R}^{\mathcal{L}_{D_c^Z\otimes B_3}}=-264\mathbf{R}^{\mathcal{L}_{D^Z_c}},\\\notag
&\mathbf{R}^{\mathcal{L}_{D_c^Z\otimes B_4}}=-117288\mathbf{R}^{\mathcal{L}_{D^Z_c}}.
\end{align}
\end{thm}
\begin{thm}
When ${\rm dim}TZ=17$ and $p_1(TZ)=3p_1(L_R)$, we have
\begin{align}
&\mathbf{R}^{\mathcal{G}_{D_c^Z\otimes B_1}}=-264\mathbf{R}^{\mathcal{G}_{D^Z_c}},\\\notag
&\mathbf{R}^{\mathcal{G}_{D_c^Z\otimes B_2}}=-117288\mathbf{R}^{\mathcal{G}_{D^Z_c}}.
\end{align}
When ${\rm dim}TZ=17$ and $p_1(TZ)=p_1(L_R)$, we have
\begin{align}
&\mathbf{R}^{\mathcal{G}_{D_c^Z\otimes B_3}}=-264\mathbf{R}^{\mathcal{G}_{D^Z_c}},\\\notag
&\mathbf{R}^{\mathcal{G}_{D_c^Z\otimes B_4}}=-117288\mathbf{R}^{\mathcal{G}_{D^Z_c}}.
\end{align}
\end{thm}

\begin{thm}
When ${\rm dim}TZ=19$ and $p_1(TZ)=3p_1(L_R)$, we have
\begin{align}
&\overline{\eta}(D_c^Z\otimes B_1)=-264\overline{\eta}(D^Z_c)+c_{13},\\\notag
&\overline{\eta}(D_c^Z\otimes B_2)=-117288\overline{\eta}(D^Z_c)+c_{14}.
\end{align}
When ${\rm dim}TZ=19$ and $p_1(TZ)=p_1(L_R)$, we have
\begin{align}
&\overline{\eta}(D_c^Z\otimes B_3)=-264\overline{\eta}(D^Z_c)+c_{15},\\\notag
&\overline{\eta}(D_c^Z\otimes B_4)=-117288\overline{\eta}(D^Z_c)+c_{16},
\end{align}
where $c_j,~1\leq j\leq 16$ is a constant.
\end{thm}

\section{The two-variable elliptic genus for fibre bundles}
Let $(TZ,J)$ be a $2d-2$-dimensional vector bundle with the almost complex structure $J$ and $T^{1,0}Z$ be the holomorphic tangent bundle in the sense of $J$ and $T^{1,0*}Z$ is the dual of $T^{1,0}Z$. Let $W$ denote a complex $l$-dimensional vector bundle on $M$. Denote the first Chern classes of $T^{1,0}Z$ and $W$ by $c_1(TZ)$ and $c_1(W)$. We denote by
$2\pi \sqrt{-1}x_i~(1\leq i\leq d-1)$ and $2\pi \sqrt{-1}w_j~(1\leq j\leq l)$ respectively the formal Chern roots of $T^{1,0}Z$ and $W$. Then the Todd form of $TZ$ is defined by
\begin{equation}
  {\rm Td}(TZ):=\prod_{i=1}^{d-1}\frac{2\pi \sqrt{-1}x_i}{1-e^{-2\pi \sqrt{-1}x_i}}.
\end{equation}
Let $(\tau,z)\in \mathcal{H}\times \mathcal{C}$ where $\mathcal{H}$ is the upper half plane and $\mathcal{C}$ is the complex plane. Let $y=e^{2\pi \sqrt{-1}z}$ and $q= e^{2\pi \sqrt{-1}\tau}$
and $c:=\prod_{j=1}^{\infty}(1-q^j)$.\\

Let \begin{align}
{\rm E}(TZ,W,\tau,z):=&c^{2(d-1-l)}y^{-\frac{l}{2}} \bigotimes _{n=1}^{\infty}
\left( \wedge_{-yq^{n-1}}(W^*)\otimes\wedge_{-y^{-1}q^{n}}(W)\right)\\\notag
&\bigotimes\left(\bigotimes _{n=1}^{\infty}S_{q^n}(T^{1,0*}Z)\bigotimes _{n=1}^{\infty}S_{q^n}(T^{1,0}Z)\right),
\end{align}
where the coefficient $c^{2(d-1-l)}$ is different from the coefficient $c^{2(d-l)}$ in \cite{Li}

\begin{defn}
The two variable elliptic genus of $M\rightarrow B$ with respect to $W$, which we denote by ${\rm Ell}(TZ,W,\tau,z)$ is defined by
\begin{equation}
{\rm Ell}(TZ,W,\tau,z):=\left\{exp(\frac{c_1(W)-c_1(T^{1,0}Z)}{2})Td(TZ){\rm ch}(E(TM,W,\tau,z))\right\}^{(2d)},
\end{equation}
\end{defn}

Using the same calculations as in Lemma 3.4 in \cite{Li} and
 we have
\begin{lem}
\begin{align}
{\rm Ell}(M,W,\tau,z)=\left(\eta(\tau)^{3(d-1-l)}\prod_{i=1}^{d-1}\frac{2\pi \sqrt{-1}x_i}{\theta(\tau,x_i)}\prod_{j=1}^l\theta
(\tau,w_j-z)\right)^{(2d)},
\end{align}
where
 \begin{equation}
 \eta(\tau):=q^{\frac{1}{24}}\cdot c=q^{\frac{1}{24}}\prod_{j=1}^{\infty}(1-q^j).
   \end{equation}
\end{lem}

Similarly to Theorem 3.2 in \cite{Li}, we have
\begin{thm}
If $c_1(W)=c_1(T^{1,0}Z)=0$ and the first Pontrjagin classes $p_1(TZ)=p_1(W)$, then
the elliptic genus ${\rm Ell}(TZ,W,\tau,z)$ is a weak Jacobi form of weight $d-l$ and index $\frac{l}{2}$.
\end{thm}

Similar to Proposition 3.6 in \cite{Li} (our coefficients of $a_n$ are different from the coefficients in \cite{Li}) , we have
\begin{prop}Let $c_1(W)=c_1(T^{1,0}Z)=0$ and the first Pontrjagin classes $p_1(M)=p_1(W)$, then the series $a_n(TZ,W,\tau)$ determined by
\begin{align}
{\rm exp}(-4\pi^2lE_2(\tau)z^2){\rm Ell}(TZ,W,\tau,z)=\sum_{n\geq 0}a_n(TZ,W,\tau)\cdot z^n,
\end{align}
are modular forms of weight $d-l+n$ over $SL(2,Z)$. Furthermore, the first five series of $a_n(TZ,W,\tau)$ are of the following form:
\begin{align}
&a_0(TZ,W,\tau)=\left\{Td(TM){\rm ch}(\wedge_{-1}(W^*))\right\}^{(2d)}+q\left\{Td(M){\rm ch}(\wedge_{-1}(W^*)){\rm ch}(A_2)\right\}^{(2d)}\\\notag
&
+q^2\left\{Td(M){\rm ch}(\wedge_{-1}(W^*)){\rm ch}(A_3)\right\}^{(2d)}+O(q^3),
\end{align}
where
\begin{align}
A_2=T^{1,0}+T^{1,0*}-2(d-1-l)-W-W^*,
\end{align}
and
\begin{align}
&A_3=S^2T^{1,0}+T^{1,0*}\otimes T^{1,0}+S^2T^{1,0*}+\wedge^2W^*+\wedge^2W+W^*\otimes W_0\\\notag
&+[2(d-1-l)-1](W+W^*-T^{1,0}-T^{1,0*})-(W+W^*)\otimes (T^{1,0}+T^{1,0*})\\\notag
&+
(d-1-l)(2d-2l-5),
\end{align}
\begin{align}
&a_1(TZ,W,\tau)=\left\{2\pi\sqrt{-1}Td(TZ){\rm ch}(\sum_{p=0}^l(-1)^{p}(p-\frac{l}{2})\wedge^pW^*
)\right\}^{(2d)}\\\notag
&+q\left\{Td(M){\rm ch}(A_4)\right\}^{(2d)}+O(q^2),
\end{align}
where
\begin{align}
&A_4=-2\pi\sqrt{-1}\wedge_{-1}W^*\otimes(W+W^*)\\\notag
&+\pi\sqrt{-1}[2\sum_{j=1}^l(-1)^j\wedge^jW^*-l\wedge_{-1}W^*]\otimes
[-2(d-1-l)-(W+W^*)+T^{1,0}Z+T^{1,0*}Z],
\end{align}
\begin{align}
&a_2(TZ,W,\tau)=\left\{Td(TZ){\rm ch}(A_5
)\right\}^{(2d)}
+O(q),
\end{align}
where $A_5=-2\pi^2\sum_{p=0}^l(-1)^{p}(p-\frac{l}{2})^2\wedge^{p}W^*+
\frac{l}{6}\pi^2(\sum_{p=0}^l(-1)^{p}
\wedge^{p}W^*).$
\begin{align}
&a_3(TZ,W,\tau)=\left\{Td(TZ)
{\rm ch}(A_6)\right\}^{(2d)}
+O(q),
\end{align}
where $A_6=\frac{4}{3}\pi^3(\sqrt{-1})^3\sum_{p=0}^l(-1)^p(p-\frac{l}{2})^3\wedge^{p}W^*
+\frac{\sqrt{-1}l}{3}\pi^3\sum_{p=0}^l(-1)^{p}
(p-\frac{l}{2})\wedge^{p}W^*.$
\begin{align}
&a_4(M,W,\tau)=\left\{Td(TZ)
{\rm ch}(A_7)\right\}^{(2d)}
+O(q),
\end{align}
where \begin{align}
&A_7=\frac{2}{3}\pi^4\sum_{p=0}^l(-1)^p(p-\frac{l}{2})^4\wedge^{p}W^*
-\frac{l^2}{3}\pi^4(\sum_{p=0}^l(-1)^p(p-\frac{l}{2})^2\wedge^{p}W^*)\\\notag
&+\frac{l^2}{72}\pi^4
(\sum_{p=0}^l(-1)^p\wedge^{p}W^*).
\end{align}
\end{prop}
Since there are no $SL(2,Z)$ modular forms with the odd weight or the non zero weight $\leq 2$, we have
\begin{prop}
Let ${\rm dim}TZ=2d-2$, $c_1(W)=c_1(T^{1,0}Z)=0$ and the first Pontrjagin classes $p_1(TZ)=p_1(W)$, then\\
1)if either $d-l$ is odd or $d-l\leq 2$ but $d-l\neq 0$, then
\begin{align}
&\mathbf{R}^{\mathcal{L}_{D^Z\otimes \wedge_{-1}W^*}}=\mathbf{R}^{\mathcal{L}_{D^Z\otimes \wedge_{-1}W^*\otimes A_2}}=0,\\\notag
&\mathbf{R}^{\mathcal{L}_{D^Z\otimes \wedge_{-1}W^*\otimes A_3}}=0,
\end{align}
2)if either $d-l$ is even or $d-l\leq 1$ but $d-l\neq -1$, then
\begin{align}
&\mathbf{R}^{\mathcal{L}_{D^Z\otimes(\sum_{p=0}^l(-1)^{p}(p-\frac{l}{2})\wedge^pW^*)}}
=0,\\\notag
&\mathbf{R}^{\mathcal{L}_{D^Z\otimes A_4}}=0,
\end{align}
3)if either $d-l$ is odd or $d-l\leq 0$ but $d-l\neq -2$, then
\begin{align}
&\mathbf{R}^{\mathcal{L}_{D^Z \otimes A_5}}=0,
\end{align}
4)if either $d-l$ is even or $d-l\leq -1$ but $d-l\neq -3$, then
\begin{align}
&\mathbf{R}^{\mathcal{L}_{D^Z\otimes A_6}}=0,
\end{align}
5)if either $d-l$ is odd or $d-l\leq -2$ but $d-l\neq -4$, then
\begin{align}
&\mathbf{R}^{\mathcal{L}_{D^Z\otimes A_7}}=0,
\end{align}
\end{prop}
\begin{thm}Let ${\rm dim}TZ=2d-2$, $c_1(W)=c_1(T^{1,0}Z)=0$ and the first Pontrjagin classes $p_1(TZ)=p_1(W)$, then\\
1)if $d-l=4$, then
\begin{align}
&\mathbf{R}^{\mathcal{L}_{D^Z\otimes \wedge_{-1}W^*\otimes A_2}}=240\mathbf{R}^{\mathcal{L}_{D^Z\otimes \wedge_{-1}W^*}},\\\notag
&\mathbf{R}^{\mathcal{L}_{D^Z\otimes \wedge_{-1}W^*\otimes A_3}}=2160\mathbf{R}^{\mathcal{L}_{D^Z\otimes \wedge_{-1}W^*}}.
\end{align}

2)if $d-l=6$, then
\begin{align}
&\mathbf{R}^{\mathcal{L}_{D^Z\otimes \wedge_{-1}W^*\otimes A_2}}=-504\mathbf{R}^{\mathcal{L}_{D^Z\otimes \wedge_{-1}W^*}},\\\notag
&\mathbf{R}^{\mathcal{L}_{D^Z\otimes \wedge_{-1}W^*\otimes A_3}}=-16632\mathbf{R}^{\mathcal{L}_{D^Z\otimes \wedge_{-1}W^*}}.
\end{align}

3)if $d-l=8$, then
\begin{align}
&\mathbf{R}^{\mathcal{L}_{D^Z\otimes \wedge_{-1}W^*\otimes A_2}}=480\mathbf{R}^{\mathcal{L}_{D^Z\otimes \wedge_{-1}W^*}},\\\notag
&\mathbf{R}^{\mathcal{L}_{D^Z\otimes \wedge_{-1}W^*\otimes A_3}}=61920\mathbf{R}^{\mathcal{L}_{D^Z\otimes \wedge_{-1}W^*}}.
\end{align}
4)if $d-l=10$, then
\begin{align}
&\mathbf{R}^{\mathcal{L}_{D^Z\otimes \wedge_{-1}W^*\otimes A_2}}=-264\mathbf{R}^{\mathcal{L}_{D^Z\otimes \wedge_{-1}W^*}},\\\notag
&\mathbf{R}^{\mathcal{L}_{D^Z\otimes \wedge_{-1}W^*\otimes A_3}}=-135432\mathbf{R}^{\mathcal{L}_{D^Z\otimes \wedge_{-1}W^*}}.
\end{align}
\end{thm}
Similarly, by $a_1(TZ,W,\tau)$, we have
\begin{thm}Let ${\rm dim}TZ=2d-2$, $c_1(W)=c_1(T^{1,0}Z)=0$ and the first Pontrjagin classes $p_1(TZ)=p_1(W)$, then\\
1)if $d-l=3$, then
\begin{align}
&\mathbf{R}^{\mathcal{L}_{D^Z\otimes \wedge_{-1}W^*\otimes A_4}}=240\mathbf{R}^{\mathcal{L}_{D^Z\otimes (\sum_{p=0}^l(-1)^{p}(p-\frac{l}{2})\wedge^pW^*)}},\\\notag
\end{align}
2)if $d-l=5$, then
\begin{align}
&\mathbf{R}^{\mathcal{L}_{D^Z\otimes \wedge_{-1}W^*\otimes A_4}}=-504\mathbf{R}^{\mathcal{L}_{D^Z\otimes (\sum_{p=0}^l(-1)^{p}(p-\frac{l}{2})\wedge^pW^*)}},\\\notag
\end{align}
3)if $d-l=7$, then
\begin{align}
&\mathbf{R}^{\mathcal{L}_{D^Z\otimes \wedge_{-1}W^*\otimes A_4}}=480\mathbf{R}^{\mathcal{L}_{D^Z\otimes (\sum_{p=0}^l(-1)^{p}(p-\frac{l}{2})\wedge^pW^*)}},\\\notag
\end{align}
4)if $d-l=9$, then
\begin{align}
&\mathbf{R}^{\mathcal{L}_{D^Z\otimes \wedge_{-1}W^*\otimes A_4}}=-264\mathbf{R}^{\mathcal{L}_{D^Z\otimes (\sum_{p=0}^l(-1)^{p}(p-\frac{l}{2})\wedge^pW^*)}},\\\notag
\end{align}
\end{thm}
Let $TZ$ be a $2d-3$-dimensional spin vector bundle.
Let $\pm 2\pi \sqrt{-1}x_i~(1\leq i\leq d-2),~0$ and $2\pi \sqrt{-1}w_j~(1\leq j\leq l)$ respectively the formal Chern roots of $T_CZ$ and $W$. Then
Let \begin{align}
\widetilde{{\rm E}}(TZ,W,\tau,z):=&c^{2(d-l)-3}y^{-\frac{l}{2}} \bigotimes _{n=1}^{\infty}
\left( \wedge_{-yq^{n-1}}(W^*)\otimes\wedge_{-y^{-1}q^{n}}(W)\right)\\\notag
&\bigotimes\left(\bigotimes _{n=1}^{\infty}S_{q^n}(T_CZ)\right),
\end{align}
where the coefficient $c^{2(d-1)-3}$ is different from the coefficient $c^{2(d-l)-2}$ in ${\rm dim }TZ=2d-2.$
\begin{defn}
The two variable elliptic genus of $M\rightarrow B$ with respect to $W$, which we denote by $\widetilde{{\rm Ell}}(TZ,W,\tau,z)$ is defined by
\begin{equation}
\widetilde{{\rm Ell}}(TZ,W,\tau,z):=\left\{exp(\frac{c_1(W)}{2})\widehat{A}(TZ){\rm ch}(\widetilde{E}(TZ,W,\tau,z))\right\}^{(2d)},
\end{equation}
\end{defn}
Using the same calculations as in Lemma 3.4 in \cite{Li} and
 we have
\begin{lem}
\begin{align}
\widetilde{{\rm Ell}}(TZ,W,\tau,z)=\left(\eta(\tau)^{3(d-l-2)}\prod_{i=1}^{d-2}\frac{2\pi \sqrt{-1}x_i}{\theta(\tau,x_i)}\prod_{j=1}^l\theta
(\tau,w_j-z)\right)^{(2d)},
\end{align}
\end{lem}

Similarly to Theorem 3.2 in \cite{Li}, we have
\begin{lem}
If $c_1(W)=0$ and the first Pontrjagin classes $p_1(TZ)=p_1(W)$, then
the elliptic genus $\widetilde{{\rm Ell}}(TZ,W,\tau,z)$ is a weak Jacobi form of weight $d-l$ and index $\frac{l}{2}$.
\end{lem}

Similar to Proposition 3.6 in \cite{Li} (our coefficients of $\widetilde{a_n}$ are different from the coefficients in $a_n$) , we have
\begin{prop}If $c_1(W)=0$ and the first Pontrjagin classes $p_1(TZ)=p_1(W)$, then the series $\widetilde{a_n}(TZ,W,\tau)$ determined by
\begin{align}
{\rm exp}(-4\pi^2lE_2(\tau)z^2)\widetilde{{\rm Ell}}(TZ,W,\tau,z)=\sum_{n\geq 0}\widetilde{a_n}(TZ,W,\tau)\cdot z^n,
\end{align}
are modular forms of weight $d-l+n$ over $SL(2,Z)$. Furthermore, the first five series of $\widetilde{a_n}(M,W,g,\tau)$ are of the following form:
\begin{align}
&\widetilde{a_0}(TZ,W,\tau)=\left\{\widehat{A}(TZ){\rm ch}(\wedge_{-1}(W^*))\right\}^{(2d)}+q\left\{\widehat{A}(TZ){\rm ch}(\wedge_{-1}(W^*)){\rm ch}(A_8)\right\}^{(2d)}\\\notag
&
+q^2\left\{\widehat{A}(TZ){\rm ch}(\wedge_{-1}(W^*)){\rm ch}(A_9)\right\}^{(2d)}+O(q^3),
\end{align}
where
\begin{align}
A_8=T_CZ-2(d-1)+3-W-W^*,
\end{align}
and
\begin{align}
&A_9=S^2T_CZ+\wedge^2W^*+\wedge^2W+W^*\otimes W\\\notag
&+[2(d-1)-4](W+W^*-T_CZ)-(W+W^*)\otimes T_CZ\\\notag
&+
(d-l-3)(2d-2l-3),
\end{align}
\begin{align}
&\widetilde{a_1}(TZ,W,\tau)=\left\{2\pi\sqrt{-1}\widehat{A}(TZ){\rm ch}(\sum_{p=0}^l(-1)^{p}(p-\frac{l}{2})\wedge^pW^*
)\right\}^{(2d)}\\\notag
&+q\left\{\widehat{A}(TZ){\rm ch}(A_{10})\right\}^{(2d)}+O(q^2),
\end{align}
where
\begin{align}
&A_{10}=-2\pi\sqrt{-1}\wedge_{-1}W^*\otimes(W+W^*)\\\notag
&+\pi\sqrt{-1}[2\sum_{j=1}^l(-1)^j\wedge^jW^*-l\wedge_{-1}W^*]\otimes
[-2(d-1)+3-(W+W^*)+T_CZ],
\end{align}
\begin{align}
&\widetilde{a_2}(TZ,W,\tau)=\left\{\widehat{A}(TZ){\rm ch}(A_5
)\right\}^{(2d)}
+O(q),
\end{align}
\begin{align}
&\widetilde{a_3}(TZ,W,\tau)=\left\{\widehat{A}(TZ)
{\rm ch}(A_6)\right\}^{(2d)}
+O(q),
\end{align}
\begin{align}
&\widetilde{a_4}(M,W,\tau)=\left\{\widehat{A}(TZ)
{\rm ch}(A_7)\right\}^{(2d)}
+O(q).
\end{align}
\end{prop}
Since there are no $SL(2,Z)$ modular forms with the odd weight or the non zero weight $\leq 2$, we have
\begin{prop}
Let ${\rm dim}TZ=2d-3$, $c_1(W)=0$ and the first Pontrjagin classes $p_1(TZ)=p_1(W)$, then\\
1)if either $d-l$ is odd or $d-l\leq 2$ but $d-l\neq 0$, then
\begin{align}
&\mathbf{R}^{\mathcal{G}_{D^Z\otimes \wedge_{-1}W^*}}=\mathbf{R}^{\mathcal{G}_{D^Z\otimes \wedge_{-1}W^*\otimes A_8}}=0,\\\notag
&\mathbf{R}^{\mathcal{G}_{D^Z\otimes \wedge_{-1}W^*\otimes A_9}}=0,
\end{align}
2)if either $d-l$ is even or $d-l\leq 1$ but $d-l\neq -1$, then
\begin{align}
&\mathbf{R}^{\mathcal{G}_{D^Z\otimes(\sum_{p=0}^l(-1)^{p}(p-\frac{l}{2})\wedge^pW^*)}}
=0,\\\notag
&\mathbf{R}^{\mathcal{G}_{D^Z\otimes A_{10}}}=0,
\end{align}
3)if either $d-l$ is odd or $d-l\leq 0$ but $d-l\neq -2$, then
\begin{align}
&\mathbf{R}^{\mathcal{G}_{D^Z\otimes A_5}}=0,
\end{align}
4)if either $d-l$ is even or $d-l\leq -1$ but $d-l\neq -3$, then
\begin{align}
&\mathbf{R}^{\mathcal{G}_{D^Z \otimes A_6}}=0,
\end{align}
5)if either $d-l$ is odd or $d-l\leq -2$ but $d-l\neq -4$, then
\begin{align}
&\mathbf{R}^{\mathcal{G}_{D^Z\otimes A_7}}=0,
\end{align}
\end{prop}
\begin{thm}Let ${\rm dim}TZ=2d-3$, $c_1(W)=0$ and the first Pontrjagin classes $p_1(TZ)=p_1(W)$, then\\
1)if $d-l=4$, then
\begin{align}
&\mathbf{R}^{\mathcal{G}_{D^Z\otimes \wedge_{-1}W^*\otimes A_8}}=240\mathbf{R}^{\mathcal{G}_{D^Z\otimes \wedge_{-1}W^*}},\\\notag
&\mathbf{R}^{\mathcal{G}_{D^Z\otimes \wedge_{-1}W^*\otimes A_9}}=2160\mathbf{R}^{\mathcal{G}_{D^Z\otimes \wedge_{-1}W^*}}.
\end{align}

2)if $d-l=6$, then
\begin{align}
&\mathbf{R}^{\mathcal{G}_{D^Z\otimes \wedge_{-1}W^*\otimes A_8}}=-504\mathbf{R}^{\mathcal{G}_{D^Z\otimes \wedge_{-1}W^*}},\\\notag
&\mathbf{R}^{\mathcal{G}_{D^Z\otimes \wedge_{-1}W^*\otimes A_9}}=-16632\mathbf{R}^{\mathcal{G}_{D^Z\otimes \wedge_{-1}W^*}}.
\end{align}

3)if $d-l=8$, then
\begin{align}
&\mathbf{R}^{\mathcal{G}_{D^Z\otimes \wedge_{-1}W^*\otimes A_8}}=480\mathbf{R}^{\mathcal{G}_{D^Z\otimes \wedge_{-1}W^*}},\\\notag
&\mathbf{R}^{\mathcal{G}_{D^Z\otimes \wedge_{-1}W^*\otimes A_9}}=61920\mathbf{R}^{\mathcal{G}_{D^Z\otimes \wedge_{-1}W^*}}.
\end{align}
4)if $d-l=10$, then
\begin{align}
&\mathbf{R}^{\mathcal{G}_{D^Z\otimes \wedge_{-1}W^*\otimes A_8}}=-264\mathbf{R}^{\mathcal{G}_{D^Z\otimes \wedge_{-1}W^*}},\\\notag
&\mathbf{R}^{\mathcal{G}_{D^Z\otimes \wedge_{-1}W^*\otimes A_9}}=-135432\mathbf{R}^{\mathcal{G}_{D^Z\otimes \wedge_{-1}W^*}}.
\end{align}
\end{thm}
Similarly, by $\widetilde{a_1}(TZ,W,\tau)$, we have
\begin{thm}Let ${\rm dim}TZ=2d-3$, $c_1(W)=0$ and the first Pontrjagin classes $p_1(TZ)=p_1(W)$, then\\
1)if $d-l=3$, then
\begin{align}
&\mathbf{R}^{\mathcal{G}_{D^Z\otimes \wedge_{-1}W^*\otimes A_{10}}}=240\mathbf{R}^{\mathcal{G}_{D^Z\otimes (\sum_{p=0}^l(-1)^{p}(p-\frac{l}{2})\wedge^pW^*)}},\\\notag
\end{align}
2)if $d-l=5$, then
\begin{align}
&\mathbf{R}^{\mathcal{G}_{D^Z\otimes \wedge_{-1}W^*\otimes A_{10}}}=-504\mathbf{R}^{\mathcal{G}_{D^Z\otimes (\sum_{p=0}^l(-1)^{p}(p-\frac{l}{2})\wedge^pW^*)}},\\\notag
\end{align}
3)if $d-l=7$, then
\begin{align}
&\mathbf{R}^{\mathcal{G}_{D^Z\otimes \wedge_{-1}W^*\otimes A_{10}}}=480\mathbf{R}^{\mathcal{G}_{D^Z\otimes (\sum_{p=0}^l(-1)^{p}(p-\frac{l}{2})\wedge^pW^*)}},\\\notag
\end{align}
4)if $d-l=9$, then
\begin{align}
&\mathbf{R}^{\mathcal{G}_{D^Z\otimes \wedge_{-1}W^*\otimes A_{10}}}=-264\mathbf{R}^{\mathcal{G}_{D^Z\otimes (\sum_{p=0}^l(-1)^{p}(p-\frac{l}{2})\wedge^pW^*)}},\\\notag
\end{align}
\end{thm}
Let $TZ$ be a $2d-1$-dimensional spin vector bundle.
Then
Let \begin{align}
\overline{{\rm E}}(TZ,W,\tau,z):=&c^{2(d-l)-1}y^{-\frac{l}{2}} \bigotimes _{n=1}^{\infty}
\left( \wedge_{-yq^{n-1}}(W^*)\otimes\wedge_{-y^{-1}q^{n}}(W)\right)\\\notag
&\bigotimes\left(\bigotimes _{n=1}^{\infty}S_{q^n}(T_CZ)\right),
\end{align}
Similarly to the definition $\widetilde{{\rm Ell}}(TZ,W,\tau,z)$, we can define the two variable elliptic genus $\overline{{\rm Ell}}(TZ,W,\tau,z)$.
Using the same calculations as in Lemma 3.4 in \cite{Li} and
 we have
\begin{lem}
\begin{align}
\overline{{\rm Ell}}(TZ,W,\tau,z)=\left(\eta(\tau)^{3(d-l-1)}\prod_{i=1}^{d-1}\frac{2\pi \sqrt{-1}x_i}{\theta(\tau,x_i)}\prod_{j=1}^l\theta
(\tau,w_j-z)\right)^{(2d)},
\end{align}
\end{lem}
Similarly to Theorem 3.2 in \cite{Li}, we have
\begin{lem}
If $c_1(W)=0$ and the first Pontrjagin classes $p_1(TZ)=p_1(W)$, then
the elliptic genus $\overline{{\rm Ell}}(TZ,W,\tau,z)$ is a weak Jacobi form of weight $d-l$ and index $\frac{l}{2}$.
\end{lem}
\begin{prop}If ${\rm dim}TZ=2d-1$, $c_1(W)=0$ and the first Pontrjagin classes $p_1(TZ)=p_1(W)$, then the series $\overline{{a_n}}(TZ,W,\tau)$ determined by
\begin{align}
{\rm exp}(-4\pi^2lE_2(\tau)z^2)\overline{{\rm Ell}}(TZ,W,\tau,z)=\sum_{n\geq 0}\overline{a_n}(TZ,W,\tau)\cdot z^n,
\end{align}
are modular forms of weight $d-l+n$ over $SL(2,Z)$. Furthermore, the first five series of $\overline{a_n}(M,W,g,\tau)$ are of the following form:
\begin{align}
&\overline{a_0}(TZ,W,\tau)=\left\{\widehat{A}(TZ){\rm ch}(\wedge_{-1}(W^*))\right\}^{(2d)}
+q\left\{\widehat{A}(TZ){\rm ch}(\wedge_{-1}(W^*)){\rm ch}(A_{11})\right\}^{(2d)}\\\notag
&
+q^2\left\{\widehat{A}(TZ){\rm ch}(\wedge_{-1}(W^*)){\rm ch}(A_{12})\right\}^{(2d)}+O(q^3),
\end{align}
where
\begin{align}
A_{11}=T_CZ-2(d-1)+1-W-W^*,
\end{align}
and
\begin{align}
&A_{12}=S^2T_CZ+\wedge^2W^*+\wedge^2W+W^*\otimes W\\\notag
&+[2(d-1)-2](W+W^*-T_CZ)-(W+W^*)\otimes T_CZ\\\notag
&+
(d-l-2)(2d-2l-1),
\end{align}
\begin{align}
&\overline{a_1}(TZ,W,\tau)=\left\{2\pi\sqrt{-1}\widehat{A}(TZ){\rm ch}(\sum_{p=0}^l(-1)^{p}(p-\frac{l}{2})\wedge^pW^*
)\right\}^{(2d)}\\\notag
&+q\left\{\widehat{A}(TZ){\rm ch}(A_{13})\right\}^{(2d)}+O(q^2),
\end{align}
where
\begin{align}
&A_{13}=-2\pi\sqrt{-1}\wedge_{-1}W^*\otimes(W+W^*)\\\notag
&+\pi\sqrt{-1}[2\sum_{j=1}^l(-1)^j\wedge^jW^*-l\wedge_{-1}W^*]\otimes
[-2(d-1)+1-(W+W^*)+T_CZ],
\end{align}
\begin{align}
&\overline{a_2}(TZ,W,\tau)=\left\{\widehat{A}(TZ){\rm ch}(A_5
)\right\}^{(2d)}
+O(q),
\end{align}
\begin{align}
&\overline{a_3}(TZ,W,\tau)=\left\{\widehat{A}(TZ)
{\rm ch}(A_6)\right\}^{(2d)}
+O(q),
\end{align}
\begin{align}
&\overline{a_4}(M,W,\tau)=\left\{\widehat{A}(TZ)
{\rm ch}(A_7)\right\}^{(2d)}
+O(q).
\end{align}
\end{prop}
Since there are no $SL(2,Z)$ modular forms with the odd weight or the non zero weight $\leq 2$, we have
\begin{prop}
Let ${\rm dim}TZ=2d-1$, $c_1(W)=0$ and the first Pontrjagin classes $p_1(TZ)=p_1(W)$, then\\
1)if either $d-l$ is odd or $d-l\leq 2$ but $d-l\neq 0$, then
\begin{align}
&\overline{\eta}(D^Z\otimes \wedge_{-1}W^*)=c_1,~\overline{\eta}(D^Z\otimes \wedge_{-1}W^*\otimes A_{11})=c_2,\\\notag
&\overline{\eta}(D^Z\otimes \wedge_{-1}W^*\otimes A_{12})=c_3,
\end{align}
2)if either $d-l$ is even or $d-l\leq 1$ but $d-l\neq -1$, then
\begin{align}
&\overline{\eta}(D^Z\otimes(\sum_{p=0}^l(-1)^{p}(p-\frac{l}{2})\wedge^pW^*))
=c_4,\\\notag
&\overline{\eta}(D^Z\otimes A_{13})=c_5,
\end{align}
3)if either $d-l$ is odd or $d-l\leq 0$ but $d-l\neq -2$, then
\begin{align}
&\overline{\eta}(D^Z\otimes  A_5)=c_6,
\end{align}
4)if either $d-l$ is even or $d-l\leq -1$ but $d-l\neq -3$, then
\begin{align}
&\overline{\eta}(D^Z\otimes A_6)=c_7,
\end{align}
5)if either $d-l$ is odd or $d-l\leq -2$ but $d-l\neq -4$, then
\begin{align}
&\overline{\eta}(D^Z\otimes A_7)=c_8,
\end{align}
where $c_j,~1\leq j\leq 8$ is a constant.
\end{prop}
\begin{thm}Let ${\rm dim}TZ=2d-1$, $c_1(W)=0$ and the first Pontrjagin classes $p_1(TZ)=p_1(W)$, then\\
1)if $d-l=4$, then
\begin{align}
&\overline{\eta}(D^Z\otimes \wedge_{-1}W^*\otimes A_{11})=240\overline{\eta}(D^Z\otimes \wedge_{-1}W^*),\\\notag
&\overline{\eta}(D^Z\otimes \wedge_{-1}W^*\otimes A_{12})=2160\overline{\eta}(D^Z\otimes \wedge_{-1}W^*).
\end{align}

2)if $d-l=6$, then
\begin{align}
&\overline{\eta}(D^Z\otimes \wedge_{-1}W^*\otimes A_{11})=-504\overline{\eta}(D^Z\otimes \wedge_{-1}W^*),\\\notag
&\overline{\eta}(D^Z\otimes \wedge_{-1}W^*\otimes A_{12})=-16632\overline{\eta}(D^Z\otimes \wedge_{-1}W^*).
\end{align}

3)if $d-l=8$, then
\begin{align}
&\overline{\eta}(D^Z\otimes \wedge_{-1}W^*\otimes A_{11})=480\overline{\eta}(D^Z\otimes \wedge_{-1}W^*),\\\notag
&\overline{\eta}(D^Z\otimes \wedge_{-1}W^*\otimes A_{12})=61920\overline{\eta}(D^Z\otimes \wedge_{-1}W^*).
\end{align}
4)if $d-l=10$, then
\begin{align}
&\overline{\eta}(D^Z\otimes \wedge_{-1}W^*\otimes A_{11})=-264\overline{\eta}(D^Z\otimes \wedge_{-1}W^*),\\\notag
&\overline{\eta}(D^Z\otimes \wedge_{-1}W^*\otimes A_{12})=-135432\overline{\eta}(D^Z\otimes \wedge_{-1}W^*).
\end{align}
\end{thm}

Similarly, by $\overline{a_1}(TZ,W,\tau)$, we have
\begin{thm}Let ${\rm dim}TZ=2d-1$, $c_1(W)=0$ and the first Pontrjagin classes $p_1(TZ)=p_1(W)$, then\\
1)if $d-l=3$, then
\begin{align}
&\overline{\eta}(D^Z\otimes \wedge_{-1}W^*\otimes A_{13})=240\overline{\eta}(D^Z\otimes (\sum_{p=0}^l(-1)^{p}(p-\frac{l}{2})\wedge^pW^*)),\\\notag
\end{align}
2)if $d-l=5$, then
\begin{align}
&\overline{\eta}(D^Z\otimes \wedge_{-1}W^*\otimes A_{13})=-504\overline{\eta}(D^Z\otimes (\sum_{p=0}^l(-1)^{p}(p-\frac{l}{2})\wedge^pW^*)),\\\notag
\end{align}
3)if $d-l=7$, then
\begin{align}
&\overline{\eta}(D^Z\otimes \wedge_{-1}W^*\otimes A_{13})=480\overline{\eta}(D^Z\otimes (\sum_{p=0}^l(-1)^{p}(p-\frac{l}{2})\wedge^pW^*)),\\\notag
\end{align}
4)if $d-l=9$, then
\begin{align}
&\overline{\eta}(D^Z\otimes \wedge_{-1}W^*\otimes A_{13})=-264\overline{\eta}(D^Z\otimes (\sum_{p=0}^l(-1)^{p}(p-\frac{l}{2})\wedge^pW^*)).\\\notag
\end{align}
\end{thm}
\section{The higher degree case}
Firstly, we recall the main theorem in \cite{MP}.
\begin{thm}(\cite{MP})
Let $M\rightarrow B$ be a closed fibre bundle with the spin closed fibre $Z$ and $A_V$ be the Bismut superconnection twisted by the vector bundle $V$. When ${\rm dim}Z=2r$, for the $j$-th Chern form $\overline{\rm str}(A^{2j}_V)_{[2j]}$, we have
\begin{align}
&\overline{\rm str}(A^{2j}_V)_{[2j]}=\frac{(-1)^jj!}{(2\pi\sqrt{-1})^{\frac{n}{2}}}(\int_Z\widehat{A}(TZ){\rm ch}(V))_{[2j]}.
\end{align}
When ${\rm dim}Z=2r+1$, for the $j$-th residue Chern form ${\rm sres}(|A|^{2j-1}_V)_{[2j-1]}$, we have
\begin{align}
&{\rm sres}(|A|^{2j-1}_V)_{[2j-1]}=\frac{(-1)^j(2j-1)!!}{(2\pi\sqrt{-1})^{\frac{n+1}{2}}2^{j-1}}(\int_Z\widehat{A}(TZ){\rm ch}(V))_{[2j-1]}.
\end{align}
where the definitions of the $j$-th Chern form $\overline{\rm str}(A^{2j}_V)_{[2j]}$ and the $j$-th residue Chern form ${\rm sres}(|A|^{2j-1}_V)_{[2j-1]}$ can be founded in \cite{MP} and \cite{PS} and $n$ is the dimension of the fibre.
\end{thm}
Let ${\rm dim}Z=2r$ and $r+j$ be even and
\begin{align}
Q_1(TZ,\tau)&=[\widehat{A}(TZ)[{\rm ch}(\triangle(TZ)\otimes \Theta_1(T_{C}Z))+2^{r}\Theta_2(T_{C}Z)+2^{r}\Theta_3(T_{C}Z)]^{(2(r+j))}.
\end{align}
$Q_1(TZ,\tau)$ is a modular form over $SL_2({\bf Z})$ with the weight $(r+j)$ and
We expand ${Q}_1(TZ,\tau)$ as follows:
\begin{align}
&{Q}_1(TZ,\tau)=\left[\widehat{A}(TZ){\rm ch}(\triangle(TZ))+2^{r+1}\widehat{A}(TZ)\right]^{(2(r+j))}\\\notag
&+\left[2\widehat{A}(TZ){\rm ch}(\triangle(TZ)){\rm ch}(\widetilde{T_CZ})+2^{r+1}\widehat{A}(TZ){\rm ch} (\widetilde{T_CZ}+\wedge^2\widetilde{T_CZ})\right]^{(2(r+j))}q\\\notag
&+\left[\widehat{A}(TZ){\rm ch}(\triangle(TZ)){\rm ch}(A_0)
+2^{r+1}\widehat{A}(TZ){\rm ch}(A_1)\right]^{(2(r+j))}q^2+\cdots.
\end{align}
\begin{thm}
When $r+j=4$, we have
\begin{align}
&\overline{\rm str}(A^{2j}_{\triangle(TZ)\otimes\widetilde{T_CZ}})_{[2j]}+2^r\overline{\rm str}(A^{2j}_{\widetilde{T_CZ}+\wedge^2\widetilde{T_CZ}})_{[2j]}
=120(\overline{\rm str}(A^{2j}_{\triangle(TZ)})_{[2j]}+2^{r+1}\overline{\rm str}(A^{2j})_{[2j]}),\\\notag
&\overline{\rm str}(A^{2j}_{\triangle(TZ)\otimes A_0})_{[2j]}+2^{r+1}\overline{\rm str}(A^{2j}_{A_1})_{[2j]}
=2160(\overline{\rm str}(A^{2j}_{\triangle(TZ)})_{[2j]}+2^{r+1}\overline{\rm str}(A^{2j})_{[2j]}).
\end{align}
When $r+j=6$, we have
\begin{align}
&\overline{\rm str}(A^{2j}_{\triangle(TZ)\otimes\widetilde{T_CZ}})_{[2j]}+2^r\overline{\rm str}(A^{2j}_{\widetilde{T_CZ}+\wedge^2\widetilde{T_CZ}})_{[2j]}
=-252(\overline{\rm str}(A^{2j}_{\triangle(TZ)})_{[2j]}+2^{r+1}\overline{\rm str}(A^{2j})_{[2j]}),\\\notag
&\overline{\rm str}(A^{2j}_{\triangle(TZ)\otimes A_0})_{[2j]}+2^{r+1}\overline{\rm str}(A^{2j}_{A_1})_{[2j]}
=-16632(\overline{\rm str}(A^{2j}_{\triangle(TZ)})_{[2j]}+2^{r+1}\overline{\rm str}(A^{2j})_{[2j]}).
\end{align}
When $r+j=8$, we have
\begin{align}
&\overline{\rm str}(A^{2j}_{\triangle(TZ)\otimes\widetilde{T_CZ}})_{[2j]}+2^r\overline{\rm str}(A^{2j}_{\widetilde{T_CZ}+\wedge^2\widetilde{T_CZ}})_{[2j]}
=480(\overline{\rm str}(A^{2j}_{\triangle(TZ)})_{[2j]}+2^{r+1}\overline{\rm str}(A^{2j})_{[2j]}),\\\notag
&\overline{\rm str}(A^{2j}_{\triangle(TZ)\otimes A_0})_{[2j]}+2^{r+1}\overline{\rm str}(A^{2j}_{A_1})_{[2j]}
=61920(\overline{\rm str}(A^{2j}_{\triangle(TZ)})_{[2j]}+2^{r+1}\overline{\rm str}(A^{2j})_{[2j]}).
\end{align}
When $r+j=10$, we have
\begin{align}
&\overline{\rm str}(A^{2j}_{\triangle(TZ)\otimes\widetilde{T_CZ}})_{[2j]}+2^r\overline{\rm str}(A^{2j}_{\widetilde{T_CZ}+\wedge^2\widetilde{T_CZ}})_{[2j]}
=-132(\overline{\rm str}(A^{2j}_{\triangle(TZ)})_{[2j]}+2^{r+1}\overline{\rm str}(A^{2j})_{[2j]}),\\\notag
&\overline{\rm str}(A^{2j}_{\triangle(TZ)\otimes A_0})_{[2j]}+2^{r+1}\overline{\rm str}(A^{2j}_{A_1})_{[2j]}
=-117288(\overline{\rm str}(A^{2j}_{\triangle(TZ)})_{[2j]}+2^{r+1}\overline{\rm str}(A^{2j})_{[2j]}).
\end{align}
\end{thm}

Let ${\rm dim}Z=2r+1$ and $r+j$ be even and
define $Q_2(TZ,\tau)$ by the same expression of $Q_1(TZ,\tau)$.
$Q_2(TZ,\tau)$ is a modular form over $SL_2({\bf Z})$ with the weight $(r+j)$ and
 ${Q}_1(TZ,\tau)$ has the same expansion of $Q_1(TZ,\tau)$.

\begin{thm}
When $r+j=4$, we have
\begin{align}
&{\rm sres}(|A|^{2j-1}_{\triangle(TZ)\otimes\widetilde{T_CZ}})_{[2j-1]}+2^r{\rm sres}(|A|^{2j-1}_{\widetilde{T_CZ}+\wedge^2\widetilde{T_CZ}})_{[2j-1]}\\\notag
&=120({\rm sres}(|A|^{2j-1}_{\triangle(TZ)})_{[2j-1]}+2^{r+1}{\rm sres}(A^{2j-1})_{[2j-1]}),\\\notag
&{\rm sres}(|A|^{2j-1}_{\triangle(TZ)\otimes A_0})_{[2j-1]}+2^{r+1}{\rm sres}(|A|^{2j-1}_{A_1})_{[2j-1]}\\\notag
&=2160({\rm sres}(|A|^{2j-1}_{\triangle(TZ)})_{[2j-1]}+2^{r+1}{\rm sres}(|A|^{2j-1})_{[2j-1]}.
\end{align}
When $r+j=6$, we have
\begin{align}
&{\rm sres}(|A|^{2j-1}_{\triangle(TZ)\otimes\widetilde{T_CZ}})_{[2j-1]}+2^r{\rm sres}(|A|^{2j-1}_{\widetilde{T_CZ}+\wedge^2\widetilde{T_CZ}})_{[2j-1]}\\\notag
&=-252({\rm sres}(|A|^{2j-1}_{\triangle(TZ)})_{[2j-1]}+2^{r+1}{\rm sres}(A^{2j-1})_{[2j-1]}),\\\notag
&{\rm sres}(|A|^{2j-1}_{\triangle(TZ)\otimes A_0})_{[2j-1]}+2^{r+1}{\rm sres}(|A|^{2j-1}_{A_1})_{[2j-1]}\\\notag
&=-16632({\rm sres}(|A|^{2j-1}_{\triangle(TZ)})_{[2j-1]}+2^{r+1}{\rm sres}(|A|^{2j-1})_{[2j-1]}.
\end{align}
When $r+j=8$, we have
\begin{align}
&{\rm sres}(|A|^{2j-1}_{\triangle(TZ)\otimes\widetilde{T_CZ}})_{[2j-1]}+2^r{\rm sres}(|A|^{2j-1}_{\widetilde{T_CZ}+\wedge^2\widetilde{T_CZ}})_{[2j-1]}\\\notag
&=480({\rm sres}(|A|^{2j-1}_{\triangle(TZ)})_{[2j-1]}+2^{r+1}{\rm sres}(A^{2j-1})_{[2j-1]}),\\\notag
&{\rm sres}(|A|^{2j-1}_{\triangle(TZ)\otimes A_0})_{[2j-1]}+2^{r+1}{\rm sres}(|A|^{2j-1}_{A_1})_{[2j-1]}\\\notag
&=61920({\rm sres}(|A|^{2j-1}_{\triangle(TZ)})_{[2j-1]}+2^{r+1}{\rm sres}(|A|^{2j-1})_{[2j-1]}.
\end{align}
When $r+j=10$, we have
\begin{align}
&{\rm sres}(|A|^{2j-1}_{\triangle(TZ)\otimes\widetilde{T_CZ}})_{[2j-1]}+2^r{\rm sres}(|A|^{2j-1}_{\widetilde{T_CZ}+\wedge^2\widetilde{T_CZ}})_{[2j-1]}\\\notag
&=-132({\rm sres}(|A|^{2j-1}_{\triangle(TZ)})_{[2j-1]}+2^{r+1}{\rm sres}(A^{2j-1})_{[2j-1]}),\\\notag
&{\rm sres}(|A|^{2j-1}_{\triangle(TZ)\otimes A_0})_{[2j-1]}+2^{r+1}{\rm sres}(|A|^{2j-1}_{A_1})_{[2j-1]}\\\notag
&=-117288({\rm sres}(|A|^{2j-1}_{\triangle(TZ)})_{[2j-1]}+2^{r+1}{\rm sres}(|A|^{2j-1})_{[2j-1]}.
\end{align}
\end{thm}
\noindent{\bf Remark 1.} We can use Theorem 4.1 to give the higher degree cases of Theorems in previous sections. We also may generalize the results in \cite{GWL} to family cases using the tricks in this paper. But we do not know whether one has the similar theorems due to Bismut-Freed and Lott for the Dai-Zhang's Toeplitz family index theorem (\cite{DZ}) or not. We can also give the higher degree cases of \cite{HL} and \cite{GL}. Thus by the Lott's theorem about the curvature of the Deligne cohomology class in \cite{Lo}, we may give another higher degree explanation.\\

\noindent{\bf Remark 2.} Motivated by Theorem 4.1, it is interesting to extend the well-known Kaster-Kalau-Walze theorem in NCG (\cite{Ka}, \cite{KW}) to the family case. That is to compute the family Wodzicki residue of $|A|^{-n+2}$ on the fibration $M$ with $n$-dimensional fibre $Z$.

\section{Acknowledgements}
The author was supported by Science and Technology Development Plan Project of Jilin Province, China: No.20260102245JC.

\vskip 1 true cm

\section{Data availability}

No data was gathered for this article.

\section{Conflict of interest}

The authors have no relevant financial or non-financial interests to disclose.

\vskip 1 true cm

\bigskip
\bigskip
\indent{Y. Wang}\\
 \indent{School of Mathematics and Statistics,
Northeast Normal University, Changchun Jilin, 130024, China }\\
\indent E-mail: {\it wangy581@nenu.edu.cn }\\

\end{document}